\documentclass[a4paper,10pt]{article}

\usepackage[utf8]{inputenc} 
\usepackage{amsmath}
\usepackage{amsbsy}
\usepackage{amssymb}
\usepackage{amscd,amsthm}
\usepackage{graphicx}
\usepackage[mathcal]{eucal}
\usepackage{verbatim}
\usepackage[english]{babel}
\usepackage[dvigs]{epsfig}
\usepackage{dsfont}
\usepackage{longtable}
\usepackage{marvosym}
\usepackage{anysize}

\newcommand{\e}{\varepsilon}

\newcommand{\tl}{\text{li}}

\newcommand{\N}{\mathds{N}}
\newcommand{\R}{\mathds{R}}

\newcommand{\p}{\phantom}
\newcommand{\q}{\quad}

\newtheorem{thm}{Theorem}[section]
\newtheorem{lem}[thm]{Lemma}
\newtheorem{kor}[thm]{Corollary}
\newtheorem{prop}[thm]{Proposition}

\theoremstyle{remark}
\newtheorem*{rema}{Remark}

\theoremstyle{definition}
\newtheorem*{defi}{Definition}

\title{New bounds for the prime counting function $\pi(x)$}
\author{Christian Axler}

\begin{document}

\maketitle

%

\begin{abstract}
In this paper we establish a number of new estimates concerning the prime counting function $\pi(x)$,  which  improve the estimates proved in the
literature. As an application, we deduce a new result concerning the existence of prime numbers in small intervals.
\end{abstract}

\section{Introduction}

\noindent
After Euclid \cite{euc} proved that there are infinitly many primes, the question arised how fast
\begin{displaymath}
\pi(x) = \sum_{p \leq x} 1
\end{displaymath}
increase as $x \to \infty$. In 1793, Gauss \cite{gau} conjectured that
\begin{displaymath}
\pi(x) \sim \text{li}(x) = \int_0^x \frac{dt}{\log t}  \q\q (x \to \infty),
\end{displaymath}
which is equivalent to
\begin{equation} \label{101}
\pi(x) \sim \frac{x}{\log x} \q\q (x \to \infty).
\end{equation}
In 1896, Hadamard \cite{had} and de la Vall\'{e}e-Poussin \cite{dlvp} proved, independently, the relation \eqref{101}, which is actually known as the
\emph{Prime Number Theorem}. A well-known asymptotic formula for $\pi(x)$ is given by
\begin{equation} \label{102}
\pi(x) = \frac{x}{\log x} + \frac{x}{\log^2 x} + \frac{2x}{\log^3 x} + \frac{6x}{\log^4 x} + \ldots + \frac{(n-1)!x}{\log^n x} + O \left(
\frac{x}{\log^{n+1} x} \right).
\end{equation}

\noindent
In this short paper we prove the following upper and lower bound for $\pi(x)$ for $n=8$ .

\begin{thm} \label{t101}
If $x > 1$, then
\begin{equation} \label{103}
\pi(x) < \frac{x}{\log x} + \frac{x}{\log^2 x} + \frac{2x}{\log^3 x} + \frac{6.35x}{\log^4 x} + \frac{24.35x}{\log^5 x} + \frac{121.75x}{\log^6 x} +
\frac{730.5x}{\log^7 x} + \frac{6801.4x}{\log^8 x}.
\end{equation}
\end{thm}

\begin{thm} \label{t102}
If $x \geq 1332450001$, then
\begin{displaymath}
\pi(x) > \frac{x}{\log x} + \frac{x}{\log^2 x} + \frac{2x}{\log^3 x} + \frac{5.65x}{\log^4 x} + \frac{23.65x}{\log^5 x} + \frac{118.25x}{\log^6 x} +
\frac{709.5x}{\log^7 x} + \frac{4966.5x}{\log^8 x}.
\end{displaymath}
\end{thm}

\noindent
Panaitopol \cite{pan3} showed another asymptotic formula for $\pi(x)$, by proving that
\begin{equation} \label{105}
\pi(x) = \frac{x}{\log x - 1 - \frac{k_1}{\log x} - \frac{k_2}{\log^2x} - \ldots - \frac{k_n(1+\alpha_n(x))}{\log^n x}}
\end{equation}
for every $n$, where $\lim_{x\to\infty}\alpha_n(x) = 0$ and positive integers $k_1, k_2, \ldots, k_n$ are given by the recurrence formula
\begin{displaymath}
k_n + 1! k_{n-1} + 2! k_{n-2} + \ldots + (n-1)! k_1 = n \cdot n!.
\end{displaymath}
For instance, we have $k_1 = 1$, $k_2 = 3$, $k_3 = 13$, $k_4 = 71$, $k_5 = 461$ and $k_6 = 3441$. In view of \eqref{105}, we find the following
estimates for $\pi(x)$ for $n=6$.

\begin{thm} \label{t103}
If $x \geq e^{3.804}$, then
\begin{equation} \label{106}
\pi(x) < \frac{x}{\log x - 1 - \frac{1}{\log x} - \frac{3.35}{\log^2x} - \frac{12.65}{\log^3x} - \frac{71.7}{\log^4 x} - \frac{466.1275}{\log^5 x} -
\frac{3489.8225}{\log^6 x}}.
\end{equation}
\end{thm}

\begin{thm} \label{t104}
If $x \geq 1332479531$, then
\begin{equation} \label{107}
\pi(x) > \frac{x}{\log x - 1 - \frac{1}{\log x} - \frac{2.65}{\log^2x} - \frac{13.35}{\log^3x} - \frac{70.3}{\log^4 x} - \frac{455.6275}{\log^5 x} -
\frac{3404.4225}{\log^6 x}}.
\end{equation}
\end{thm}

\noindent
As an application of these estimates, we obtain the following result concerning the existence of a prime number in a small interval.

\begin{thm} \label{t105}
For every $x \geq 58837$ there is a prime number $p$ such that
\begin{displaymath}
x < p \leq x \left( 1 + \frac{1.1817}{\log^3 x} \right).
\end{displaymath}
\end{thm}

\section{Skewes' number}

One of the first estimates for $\pi(x)$ is due to Gauss. In 1793, he computed that $\pi(x) < \text{li}(x)$ holds for every $2 \leq x \leq 3000000$ and
conjectured that $\pi(x) < \text{li}(x)$ holds for every $x \geq 2$. However, in 1914, Littlewood \cite{litt} proved, that $\pi(x) - \text{li}(x)$ changes
the sign infinitely many times by showing that there is a positive constant $K$ such that the sets
\begin{displaymath}
\left \{ x \geq 2 \mid \pi(x) - \tl(x) > \frac{K\sqrt{x}\log \log \log x}{\log x} \right \}
\end{displaymath}
and
\begin{displaymath}
\left \{ x \geq 2 \mid \pi(x) - \tl(x) < - \frac{K\sqrt{x}\log \log \log x}{\log x} \right \}
\end{displaymath}
are not empty and unbounded. However, Littlewood's proof is non-constructive and up to now no $x\geq 2$ is known such that $\pi(x) > \text{li}(x)$ holds.
Let $\Xi = \min \{ x \in \R_{\geq 2} \mid \pi(x) > \text{li}(x) \}$. The first upper bound for $\Xi$, which was proved without the assumption of Riemann's
hypothesis, is due to Skewes \cite{skew} in 1955, namely
\begin{displaymath}
\Xi < 10^{10^{10^{963}}}.
\end{displaymath}
The number on the right hand side is known in the literature as the \emph{Skewes' number}. In 1966, Lehman \cite{leh} improved this upper bound considerably
by showing that $\Xi < 1.65 \cdot 10^{1165}$. After some improvements the current best upper bound,
\begin{displaymath}
\Xi < e^{727.951336105} \leq 1.398 \cdot 10^{316},
\end{displaymath}
was proved by Saouter, Trudgian and Demichel \cite{std}. A lower bound is given by the calculation of Gauss, namely $\Xi > 3000000$. In 2008, Kotnik
\cite{ko} proved the following

\begin{prop} \label{prop201}
We have $\Xi > 10^{14}$.
\end{prop}

\begin{rema}
Recently, Platt and Trudgian \cite{pltr} improved Proposition \ref{prop201} by showing $\Xi > 1.39 \cdot 10^{17}$.
\end{rema}

\section{New estimates for $\pi(x)$}

Since there is no efficient algorithm for computing $\pi(x)$, we are interested in upper and lower bounds for $\pi(x)$. Up to now the sharpest estimates for
$\pi(x)$ are due to Dusart \cite{pd4}. In 2010, he proved that the inequality
\begin{equation} \label{308}
\pi(x) \leq \frac{x}{\log x} + \frac{x}{\log^2 x} + \frac{2.334x}{\log^3 x}
\end{equation}
holds for every $x \geq 2953652287$ and that
\begin{equation} \label{309}
\pi(x) \geq \frac{x}{\log x} + \frac{x}{\log^2 x} + \frac{2x}{\log^3 x}
\end{equation}
for every $x \geq 88783$. To find new estimates, we consider the so-called Chebyshev-function
\begin{displaymath}
\theta(x)= \sum_{p \leq x} \log p.
\end{displaymath}

\noindent
The following relation between $\pi(x)$ and $\theta(x)$ is well-known.

\begin{prop} \label{prop301}
If $x \geq 2$, then
\begin{equation} \label{310}
\pi(x) = \frac{\theta(x)}{\log x} + \int_{2}^{x}{\frac{\theta(t)}{t \log^{2} t}\ dt}.
\end{equation}
\end{prop}

\begin{proof}
See Apostol \cite[Theorem 4.3]{ap}.
\end{proof}

\noindent
Before we give our first new estimate for $\pi(x)$, we mention a result \cite{pd4} about the distance between $x$ and $\theta(x)$, which plays an important
role below.

\begin{prop} \label{prop302}
Let $k \in \mathds{N}$ with $k \leq 4$. Then for every $x \geq x_0(k)$,
\begin{equation} \label{311}
\vert \theta(x) - x \vert < \frac{\eta_k x}{\log^{k} x},
\end{equation}
where
\begin{center}
\begin{tabular}{|l||c|c|c|c|c|c|c|c|c|}
\hline
$k$     & $       1 $ & $         2$ & $         3$ & $4$  \\ \hline
$\eta_k$& $   0.001 $ & $      0.01$ & $      0.78$ & $1300$  \\ \hline
$x_0(k)$\rule{0mm}{4mm}& $908994923$ & $7713133853$ & $158822621$ & $2$ \\ \hline
\end{tabular}.
\end{center}
\end{prop}

\begin{proof}
See Dusart \cite{pd4}.
\end{proof}

\vspace*{1mm}
\noindent
By using Table 6.4 \& Table 6.5 from \cite{pd4}, we obtain the following result.

\begin{prop} \label{prop303}
If $x \geq e^{30}$, then
\begin{displaymath}
\vert \theta(x) - x \vert < \frac{0.35 x}{\log^3 x}.
\end{displaymath}
\end{prop}

\begin{proof}
We set $a = 3600$ and $\e_{\psi} = 6.93 \cdot 10^{-12}$. Then we have
\begin{equation} \label{312}
\frac{1.00007 (a+i)^3}{\sqrt{e^{a+i}}} + \frac{1.78 (a+i)^3}{(e^{a+i})^{2/3}} + \e_{\psi}(a+1+i)^3 < 0.35
\end{equation}
for every $0\leq i \leq 75$. By \cite{pd5}, we can choose $\e_{\psi} = 6.49 \cdot 10^{-12}$ for every $e^{3675} \leq x \leq e^{3700}$, so that the
inequality \eqref{312} holds with $\e_{\psi} = 6.49 \cdot 10^{-12}$ for every $75\leq i \leq 100$ as well. It follows from Table 6.4 and Table 6.5 in
\cite{pd4} that we can choose $\eta_3 = 0.35$ and $x_0(3) = e^{30}$ in \eqref{311}.
\end{proof}

\noindent
Now, let $k\in \N$ with $k \leq 4$ and let $\eta_k, x_1(k)$ be such that the inequality
\begin{equation} \label{313}
\vert \theta(x) - x \vert < \frac{\eta_k x}{\log^{k} x}
\end{equation}
holds for every $x \geq x_1(k)$. To prove new estimates for $\pi(x)$, Rosser \& Schoenfeld \cite{rs} introduced the following function, which plays an
important role below as well.

\begin{defi}
For every $x > 1$, we define
\begin{equation} \label{314}
J_{k,\eta_k,x_1(k)}(x) = \pi(x_1(k)) - \frac{\theta(x_1(k))}{\log x_1(k)} + \frac{x}{\log x} + \frac{\eta_k x}{\log^{k+1} x} + \int_{x_1(k)}^{x}{\left(
\frac{1}{\log^{2} t} + \frac{\eta_k}{\log^{k+2} t} \ dt \right)}.
\end{equation}
\end{defi}

\begin{prop} \label{prop304}
If $x \geq x_1(k)$, then
\begin{equation} \label{315}
J_{k,-\eta_k,x_1(k)}(x) < \pi(x) < J_{k,\eta_k,x_1(k)}(x)
\end{equation}
\end{prop}

\begin{proof}
The claim follows from \eqref{314}, \eqref{313} and \eqref{310}.
\end{proof}

\subsection{Some new upper bounds for $\pi(x)$}

We prove our first main result.

\begin{proof}[Proof of Theorem \ref{t101}]
We denote the term on the right hand side of \eqref{103} by $\delta(x)$ and set $\widehat \delta(x,y) = x \delta(y)/y$. Let $x_1 = 10^{14}$. We obtain
\begin{equation} \label{316}
\delta'(x) - J'_{3, 0.35,x_1}(x) = \frac{1687.9 \log x - 54411.2}{\log^9 x} \geq 0
\end{equation}
for every $x \geq x_1$. Since we have $\theta(x_1) \geq 99999990573246$ by \cite{pd4}, $\pi(x_1) = 3204941750802$ and $\log x_1 \leq 32.2362$, we obtain
\begin{equation} \label{317}
\pi(x_1) - \frac{\theta(x_1)}{\log x_1} \leq 102839438084.
\end{equation}
It follows that
\begin{displaymath}
\delta(x_1) - J_{3, 0.35,x_1}(x_1) \geq \widehat \delta(x_1, 32.2362) - J_{3, 0.35,x_1}(x_1) > 0.
\end{displaymath}
Using \eqref{315} und \eqref{316}, we obtain $\delta(x) > \pi(x)$ for every $x \geq x_1$.

We have
\begin{displaymath}
\delta'(x) - \text{li}'(x) = \frac{0.35 \log^5 x - 1.05 \log^4 x + 1687.9 \log x - 54411.2}{\log^9 x} \geq 0
\end{displaymath}
for every $x \geq 5 \cdot 10^5$. Using in addition $\delta(5 \cdot 10^5) - \text{li}(5 \cdot 10^5) \geq 2.4 > 0$ and Proposition \ref{prop201}, we get that
$\delta(x) > \pi(x)$ for every $5 \cdot 10^5 \leq x \leq 10^{14}$.

For every $x \geq 47$, we have $\delta'(x) \geq 0$. To obtain the required inequality for every $47 \leq x \leq 5 \cdot 10^5$, it suffices to
check with a computer, that $\delta(p_i) > \pi(p_i)$ holds for every $\pi(47) \leq i \leq \pi(5\cdot 10^5) + 1$, which is really the case.

Since $\pi(46) < \delta(46)$ and $\delta'(x) < 0$ is fulfilled for every $1 < x \leq 46$, we obtain $\delta(x) > \pi(x)$ for every $1 < x \leq 46$.

It remains to consider the case $46 < x \leq 47$. Here $\delta(x) > 15 > \pi(x)$, and the theorem is proved.
\end{proof}

\begin{rema}
The inequality in Theorem \ref{t101} improves Dusart's estimate \eqref{308} for every $x \geq e^{23.11}$.
\end{rema}

%
%

\noindent
By using Proposition \ref{prop201}, we prove our third main result.

\begin{proof}[Proof of Theorem \ref{t103}]
We denote the right hand side of Theorem \ref{t103} by $\xi(x)$. Let $x_1 = 10^{14}$ and let
\begin{displaymath}
g(t) = t^7 - t^6 - t^5 - 3.35t^4 - 12.65t^3 - 71.7t^2 - 466.1275t - 3489.8225.
\end{displaymath}
Then $g(t) > 0$ for every $t \geq 3.804$. We set
\begin{align*}
h(t) & = 29470 t^{10} + 11770 t^9 + 39068 t^8 + 164238 t^7 + 712906 t^6 + 3255002 t^5 \\
& \p{\q\q} + 12190826 t^4 + 88308 t^3 + 385090 t^2 + 846526 t - 12787805.
\end{align*}
Since $h(t) \geq 0$ for every $t \geq 1$, we obtain
\begin{equation} \label{318}
\xi'(x) - J'_{3, 0.35,x_1}(x) \geq \frac{h(\log x)}{g^2(\log x) \log^4 x} \geq 0
\end{equation}
for every $x \geq e^{3.804}$.

Let $K_1 = 102839438084, a = 32.23619$ and $b = 32.236192$. We set
\begin{align*}
f(s,t) = K_1t^7 & + (K_1 + s) t^6 + (3.35K_1 + s)t^5 + (12.65K_1 + 3s)t^4 + (71.7K_1 + 13s)t^3 \\
& + (466.1275K_1 + 72.05s)t^2 + (3489.8225K_1 + 467.3s)t + 3494.25s
\end{align*}
and obtain $f(x_1, a) \geq b^8 K_1$. Since $a \leq \log x_1 \leq b$, we have $f(x_1, \log x_1) \geq K_1 \log^8 x_1$ and therefore
\begin{align*}
x_1 \log^6 x_1 & + x_1\log^5 x_1 + 3x_1\log^4 x_1 + 13x_1\log^3 x_1 + 72.05x_1\log^2 x_1 + 467.3x_1\log x_1 + 3494.25x_1 \\
& \geq K_1 \log^8 x_1 - K_1\log^7 x_1 - K_1 \log^6 x_1 - 3.35K_1 \log^5 x_1 - 12.65K_1 \log^4 x_1 \\
& \p{\q\q} - 71.7K_1 \log^3 x_1 - 466.1275K_1 \log^2 x_1 - 3489.8225K_1 \log x_1.
\end{align*}
It immediately follows that
\begin{align*}
x_1 \log^9 x_1 & + x_1 \log^8 x_1 + 3x_1 \log^7 x_1 + 13x_1 \log^6 x_1 + 72.05x_1 \log^5 x_1 + 467.3x_1 \log^4 x_1 \\
& \phantom{\quad \quad} + 3494.25x_1 \log^3 x_1 + 25.095 x_1 \log^2 x_1 + 163.144625 x_1 \log x_1 + 1221.437875 x_1 \\
& > K_1 g(\log x_1) \log^4 x_1.
\end{align*}
Since the left hand side of the last inequality is equal to $x_1 ( \log^{10} x_1 - (\log^3 x_1 + 0.35)g(\log x_1))$, we have
\begin{displaymath}
x_1  \log^{10} x_1 > (K_1  \log^4 x_1 + x_1(\log^3 x_1 + 0.35))g(\log x_1).
\end{displaymath}
Moreover, $K_1 \geq \pi(x_1) - \theta(x_1)/\log x_1$ by \eqref{317} and $g(\log x_1) > 0$. Hence,
\begin{displaymath}
x_1  \log^{10} x_1 > \left( \left( \pi(x_1) - \frac{\theta(x_1)}{\log x_1} \right) \log^4 x_1 + x_1(\log^3 x_1 + 0.35) \right) g(\log x_1).
\end{displaymath}
We divide both sides of this inequality by $g(\log x_1) \log^4 x_1 > 0$ and, by \eqref{318} and Proposition \ref{prop303}, we get $\xi(x) > J_{3,
0.35,x_1}(x) \geq \pi(x)$ for every $x \geq x_1$.

Now let $140000  \leq x \leq x_1$. We compare $\xi(x)$ with $\tl(x)$. We set
\begin{align*}
r(t) & = 0.35 t^{11} - 1.75 t^{10} + 1.75 t^9 - 0.6 t^8 - 1.3 t^7 - 29492 t^6  \\
& \p{\q\q} - 11917 t^5 - 40316 t^4 - 155136 t^3 - 717716 t^2 - 3253405 t - 12178862.
\end{align*}
Then $r(t) \geq 0$ for every $t \geq 10.9$ and we obtain
\begin{equation} \label{319}
\xi'(x) - \tl'(x) \geq \frac{r(\log x)}{g^2(\log x) \log x} \geq 0
\end{equation}
for every $x \geq e^{10.9}$. We have $\xi(140000) - \tl(140000) > 0.0024$. Now use \eqref{319} and Proposition \ref{prop201}.

We consider the case $e^{4.53} \leq x < 140000$. We set
\begin{displaymath}
s(t) = t^8 - 2t^7 - t^6 - 4.35t^5 - 19.35t^4 - 109.65t^3 - 752.9275t^2 - 5820.46t - 20938.935.
\end{displaymath}
Since $s(t) \geq 0$ for every $t \geq 4.53$, we get
\begin{equation} \label{320}
\frac{g(\log x)^2\xi'(x)}{\log^5 x} = s(\log x) \geq 0
\end{equation}
for every $x \geq e^{4.53}$. Since $g(\log x) > 0$ for every $x \geq e^{3.804}$, using \eqref{320}, we obtain that $\xi'(x) > 0$ holds for every $x \geq
e^{4.53}$. So we check with a computer that $\xi(p_n) > \pi(p_n)$ for every $\pi(e^{4.53}) \leq n \leq \pi(140000) + 1$.

Next, let $45 \leq x < e^{4.52}$. Since we have  $s'(t) > 0$ for every $t \geq 3.48$ and $s(4.52) \leq -433$, we get $s(\log x) < 0$. Fromg \eqref{320}, it
follows that $\xi'(x) < 0$ for every $e^{3.804} \leq x \leq e^{4.52}$. Hence, $\xi(x) \geq \xi(e^{4.52}) > 26 > \pi(e^{4.52}) \geq \pi(x)$ for every
$e^{3.804} \leq x \leq e^{4.52}$.

Finally, $\xi(x) \geq 26 > \pi(x)$ for every $e^{4.52} \leq x \leq e^{4.53}$, and the theorem is proved.
\end{proof}

\begin{rema}
Theorem \ref{t103} leads to an improvement of Theorem \ref{t101} for every sufficiently large $x$.
\end{rema}

\begin{kor} \label{kor309}
For every $x \geq 21.95$, we have
\begin{displaymath}
\pi(x) < \frac{x}{\log x - 1 - \frac{1}{\log x} - \frac{3.35}{\log^2 x} - \frac{12.65}{\log^3 x} - \frac{89.6}{\log^4x}}.
\end{displaymath}
If $x \geq 14.36$, then
\begin{displaymath}
\pi(x) < \frac{x}{\log x - 1 - \frac{1}{\log x} - \frac{3.35}{\log^2 x} - \frac{15.43}{\log^3 x}}
\end{displaymath}
and for every $x \geq 9.25$, we have
\begin{displaymath}
\pi(x) < \frac{x}{\log x - 1 - \frac{1}{\log x} - \frac{3.83}{\log^2 x}}.
\end{displaymath}
If $x \geq 5.43$, then
\begin{displaymath}
\pi(x) < \frac{x}{\log x - 1 - \frac{1.17}{\log x}}.
\end{displaymath}
\end{kor}

\begin{proof}
The claim follows by comparing each term on the right-hand side with the right-hand side of \eqref{106} and with $\tl(x)$. For small $x$ we check the
inequalities with a computer. 
\end{proof}

\subsection{Some new lower bounds for $\pi(x)$}

Next, we prove the lower bounds for $\pi(x)$ .

\begin{proof}[Proof of Theorem \ref{t104}]
We denote the denominator of the right hand side of \eqref{107} by $\varphi(x)$. Then $\varphi(x) > 0$ for every $x \geq e^{3.79}$. Let $x_1 = 10^{14}$. We
set $\phi(x) = x/\varphi(x)$ and
\begin{align*}
r(t) = 28714 t^{10} & + 11244t^9 + 36367t^8 + 146093t^7 + 691057t^6 + 3101649t^5 \\
& + 11572765t^4 - 77484t^3 - 365233t^2 - 799121t + 12169597.
\end{align*}
Obviously $r(t) \geq 0$ for every $t \geq 1$. Hence
\begin{equation} \label{321}
J'_{3, -0.35, x_1}(x) - \phi'(x) \geq \frac{r(\log x)}{(\varphi(x)\log^6 x)^2 \log^5 x} \geq 0
\end{equation}
for every $x \geq e^{3.79}$. Since $\theta(10^{14}) \leq 99999990573247$ by Table 6.2 of \cite{pd4}, $\pi(10^{14}) = 3204941750802$ and $32.23619 \leq \log
10^{14} \leq 32.2362$, we get
\begin{displaymath}
\pi(x_1) - \frac{\theta(x_1)}{\log x_1} \geq 102838475779.
\end{displaymath}
Hence, by \eqref{314},
\begin{displaymath}
J_{3, -0.35, x_1}(x_1) - \phi(x_1) \geq 102838475779 + \frac{10^{14}}{32.2362} - \frac{0.35 \cdot 10^{14}}{32.23619^4} -
\frac{10^{14}}{\varphi(e^{32.23619})} \geq 322936.
\end{displaymath}
Using \eqref{321} and Proposition \ref{prop303}, we obtain $\pi(x) > \phi(x)$ for every $x \geq x_1$.

Next, let $x_2 = 8 \cdot 10^9$ and $x_2 \leq x \leq x_1$. We set
\begin{displaymath}
h(t) = -0.01t^{15}  + 0.39t^{14} - 1.78t^{13} + 1.763t^{12} + 0.033t^{11} - 2.997t^{10}.
\end{displaymath}
For every $29 \leq t \leq 33$, we get $h(t) \geq 0.443t^{12} - 2.997t^{10} > 0$. For every $23 \leq t \leq 29$, we obtain $h(t) \geq 13.723t^{12} -
2.997t^{10} > 0$. Therefore
\begin{equation} \label{322}
J'_{2, -0.01, x_2}(x) - \phi'(x) \geq \frac{h(\log x)}{(\varphi(x)\log^6 x)^2\log^4 x} \geq 0
\end{equation}
for every $e^{23} \leq x_2 \leq x \leq x_1 \leq e^{33}$. Since $\theta(x_2) \leq 7999890793$ (see Table 6.1 of \cite{pd4}), $\pi(x_2) = 367783654$ and
$22.8027 \leq \log x_2$, we obtain
\begin{displaymath}
\pi(x_2) - \frac{\theta(x_2)}{\log x_2} \geq 367783654 - \frac{7999890793}{22.8027} \geq 16952796.
\end{displaymath}
Using $22.8 \leq \log x_2 \leq 22.8028$, we get
\begin{displaymath}
J_{2, -0.01, x_2}(x_2) - \phi(x_2) \geq 16952796 + \frac{x_2}{22.8028} - \frac{0.01x_2}{22.8^3} - \frac{x_2}{\varphi(e^{22.8})} \geq 2360.
\end{displaymath}
Using \eqref{322} and Proposition \ref{prop304}, we see that the required inequality holds for every $x_2 \leq x \leq x_1$.

It remains to consider the case $1332479531 \leq x \leq x_2$. We set
\begin{displaymath}
s(t)= t^8 - 2t^7 - t^6 - 3.65t^5 - 18.65t^4 - 110.35t^3 - 736.8275t^2 - 5682.56t - 20426.535.
\end{displaymath}
Since $s(t) \geq 0$ for every $t \geq 4.6$, we obtain
\begin{displaymath}
\phi'(x) = \frac{s(\log x)\log^5 x}{(\varphi(x)\log^6 x)^2} \geq 0
\end{displaymath}
for every $x \geq e^{4.6}$. And again we use a computer to check that $\pi(p_i) \geq \phi(p_{i+1})$ is fulfilled for every $\pi(1332479531) \leq i
\leq \pi(x_2) + 1$.
\end{proof}

\noindent
Using a computer and Theorem \ref{t104}, we obtain the following weaker estimates for $\pi(x)$.

\begin{kor} \label{kor311}
If $x \geq x_0$, then
\begin{displaymath}
\pi(x) > \frac{x}{\log x - 1 - \frac{1}{\log x} - \frac{a}{\log^2x} - \frac{b}{\log^3x} - \frac{c}{\log^4 x} - \frac{d}{\log^5 x}},
\end{displaymath}
where
\begin{center}
\begin{tabular}{|l||c|c|c|c|c|c|c|c|c|}
\hline
$a$     & $    2.65  $ & $    2.65  $ & $      2.65$ & $      2.65$ & $    2.65  $ & $    2.65  $ & $    2.65  $ \\ \hline
$b$     & $   13.35  $ & $   13.35  $ & $     13.35$ & $     13.35$ & $   13.35  $ & $   13.1   $ & $   11.6   $ \\ \hline
$c$     & $   70.3   $ & $   70.3   $ & $     45   $ & $     34   $ & $    5     $ & $    0     $ & $    0     $ \\ \hline
$d$     & $  276     $ & $   69     $ & $      0   $ & $      0   $ & $    0     $ & $    0     $ & $    0     $ \\ \hline
$x_{0}$ & $1245750347$ & $ 909050897$ & $ 768338551$ & $ 547068751$ & $ 374123969$ & $ 235194097$ & $ 166219973$ \\ \hline
\hline
$a$     & $    2.65  $ & $    2.65  $ & $    2.65  $ & $    2.62  $ & $    2.1   $ & $    1     $ & $    0     $ \\ \hline
$b$     & $    8.6   $ & $    7.7   $ & $    4.6   $ & $    0     $ & $    0     $ & $    0     $ & $    0     $ \\ \hline
$c$     & $    0     $ & $    0     $ & $    0     $ & $    0     $ & $    0     $ & $    0     $ & $    0     $ \\ \hline
$d$     & $    0     $ & $    0     $ & $    0     $ & $    0     $ & $    0     $ & $    0     $ & $    0     $ \\ \hline
$x_{0}$ & $  93811339$ & $  65951927$ & $  38168363$ & $  16590551$ & $  6690557 $ & $   1772201$ & $   468049 $ \\ \hline
\end{tabular}
\end{center}
\vspace{1mm}
\end{kor}

\begin{proof}
By comparing each right hand side with the right hand side of \eqref{107}, we see that each inequality holds for every $x \geq 1332479531$. For smaller $x$
we check the asserted inequalities using a computer.
\end{proof}

\noindent
Now we prove Theorem \ref{t102} by using Theorem \ref{t104}.

\begin{proof}[Proof of Theorem \ref{t102}]
For $y>0$ we set
\begin{displaymath}
R(y) = 1 + \frac{1}{y} + \frac{2}{y^2} + \frac{5.65}{y^3} + \frac{23.65}{y^4} + \frac{118.25}{y^5} + \frac{709.5}{y^6} + \frac{4966.5}{y^7}
\end{displaymath}
and
\begin{displaymath}
S(y) = y - 1 - \frac{1}{y} - \frac{2.65}{y^2} - \frac{13.35}{y^3} - \frac{70.3}{y^4} - \frac{455.6275}{y^5} - \frac{3404.4225}{y^6}.
\end{displaymath}
Then $S(y) > 0$ for every $y \geq 3.79$ and $y^{13}R(y)S(y) = y^{14} - T(y)$, where
\begin{align*}
T(y) & = 11017.9625y^6 + 19471.047875y^5 + 60956.6025y^4 + 250573.169y^3 \\
& \p{\q\q} + 1074985.621875y^2 + 4678311.7425y + 16908064.34625.
\end{align*}
Using Theorem \ref{t104}, we get
\begin{displaymath}
\pi(x) > \frac{x}{S(\log x)} > \frac{x}{S(\log x)} \left( 1 - \frac{T(\log x)}{\log^{14} x} \right) = \frac{x R(\log x)}{\log x}
\end{displaymath}
for every $x \geq 1332479531$. So it remains to obtain the required inequality for every $1332450001 \leq x \leq 1332479531$. Let
\begin{displaymath}
U(x) = \frac{x R(\log x)}{\log x}
\end{displaymath}
and $u(y)= y^8 - 0.35y^5 + 1.05y^4 - 39732$. Since $u(y) \geq 0$ for every $y \geq 3.8$, it follows that $U'(x) = u(\log x)/\log^9 x \geq 0$ for every $x
\geq e^{3.8}$. So we use a computer to check that the inequality $\pi(p_i) > U(p_{i+1})$ holds for every $\pi(1332450001) \leq i \leq \pi(1332479531)$.
\end{proof}

\begin{rema}
Obviously, Theorem \ref{t102} yields an improvement of Dusart's estimate \eqref{309}.
\end{rema}

\section{On the existence of prime numbers in short intervals}

Let $a,b\in \R$ and
\begin{displaymath}
z_1(a) = \min \left \{ k \in \N \mid \pi(x) > \frac{x}{\log x - 1 - \frac{1}{\log x} - \frac{a}{\log^2 x}} \; \text{for every} \; x \geq k \right \} \,
\in \N \cup \{
\infty \}
\end{displaymath}
as well as
\begin{displaymath}
z_2(b) = \min \left \{ k \in \N \mid \pi(x) < \frac{x}{\log x - 1 - \frac{1}{\log x} - \frac{b}{\log^2 x}} \; \text{for every} \; x \geq k \right \} \, \in
\N \cup \{
\infty \}.
\end{displaymath}

\begin{lem} \label{lem401}
Let $z_0 \in \R \cup \{- \infty\}$ and let $c \colon \, (z_0, \infty) \to [1,\infty)$ be a map. Then,
\begin{align*}
\pi(c(x)x) - \pi(x) & > \frac{x((c(x)-1)(\log x - 1 - \frac{1}{\log x}) - \log c(x) - \frac{c(x)\log c(x) + bc(x) - a}{\log^2 x})}{(\log (c(x)x) - 1 -
\frac{1}{\log (c(x)x)} - \frac{a}{\log^2 (c(x)x)})(\log x - 1 - \frac{1}{\log x} - \frac{b}{\log^2 x})} \\
& \p{\q\q} - \frac{x(\frac{2bc(x)\log c(x)}{\log^3 x} + \frac{bc(x)\log^2 c(x)}{\log^4 x})}{(\log (c(x)x) - 1 - \frac{1}{\log (c(x)x)} - \frac{a}{\log^2
(c(x)x)})(\log x - 1 - \frac{1}{\log x} - \frac{b}{\log^2 x})}
\end{align*}
for every $x \geq \max \{ \lfloor z_0 \rfloor + 1,z_2(b),z_3(a)\}$, where $z_3(a) = \min \{ k \in \N \mid k\,c(k) \geq z_1(a)\}$.
\end{lem}

\begin{proof}
We have
\begin{align*}
\pi(c(x)x) - \pi(x) & > \frac{c(x)x}{\log (c(x)x) - 1 - \frac{1}{\log (c(x)x)} - \frac{a}{\log^2 (c(x)x)}} - \frac{x}{\log x - 1 - \frac{1}{\log x} -
\frac{b}{\log^2 x}} \\
& = x \; \frac{(c(x)-1)(\log x - 1) - \log c(x) - \frac{c(x)-1}{\log (c(x)x)} - \frac{c(x)\log c(x)}{\log x \log (c(x)x)} - \frac{bc(x) -
a}{\log^2(c(x)x)}}{(\log (c(x)x) - 1 - \frac{1}{\log (c(x)x)} - \frac{a}{\log^2 (c(x)x)})(\log x - 1 - \frac{1}{\log x} - \frac{b}{\log^2 x})} \\
& \p{\q\q} - x\; \frac{\frac{2bc(x)\log c(x)}{\log x \log^2(c(x)x)} + \frac{bc(x)\log^2 c(x)}{\log^2 x \log^2 (c(x)x)}}{(\log (c(x)x) - 1 - \frac{1}{\log
(c(x)x)} - \frac{a}{\log^2 (c(x)x)})(\log x - 1 - \frac{1}{\log x} - \frac{b}{\log^2 x})}.
\end{align*}
Since $c(x) \geq 1$, our lemma is proved
\end{proof}

\noindent
Before proving Theorem \ref{t105}, we mention two results on the existence of prime numbers in short intervals. The first result is due to Ramar\'{e} and
Saouter \cite{rasa}. 

\begin{prop} \label{p402}
For every $x \geq 10726905041$ there exists a prime number $p$ such that
\begin{displaymath}
x < p \leq x \left( 1 + \frac{1}{28313999} \right).
\end{displaymath}
\end{prop}

\noindent
In 2014, Kadiri and Lumley \cite[Table 2]{kl} found a series of improvements of Proposition \ref{p402}. For the proof of Theorem \ref{t105}, we need the
following result which easily follows from the last row of Table 2 in \cite{kl}.

\begin{prop} \label{p403}
For every $x \geq e^{150}$ there exists a prime number $p$ such that
\begin{displaymath}
x < p \leq x \left( 1 + \frac{1}{2442159713} \right).
\end{displaymath}
\end{prop}

\noindent
Also in 2014, Trudgian \cite{trud} proved the following

\begin{prop} \label{p404}
For every $x \geq 2898239$ there exists a prime number $p$ such that
\begin{displaymath}
x < p \leq x \left( 1 + \frac{1}{111 \log^2 x} \right).
\end{displaymath}
\end{prop}

\noindent
Now we prove Theorem \ref{t105}, which leads to an improvement of Proposition \ref{p404} for every $x \geq e^{131.1687}$.

\begin{proof}[Proof of Theorem \ref{t105}]
We set $a=2.65$ and $b=3.83$. By Corollary \ref{kor311} and Corollary \ref{kor309}, we obtain $z_1(a) \leq 38168363$ and $z_2(b) = 10$. As in the
proof of Theorem \ref{t104}, we check with a computer that $z_1(a) = 36917641$. Further, we set
\begin{displaymath}
c(x) = 1 + \frac{1.1817}{\log^3 x}
\end{displaymath}
and $z_0 = 1$. Then $z_3(a) = 36909396$. We consider the function
\begin{align*}
g(x) & = 0.0017x^2 - 2.3634x - 1.1817 - \frac{5.707611}{x} - \frac{9.051822}{x^2} - \frac{1.39641489}{x^4} \\
& \p{\q\q} - \frac{10.6965380574}{x^5} - \frac{5.3482690287}{x^6} - \frac{6.32004951121479}{x^9}.
\end{align*}
and we get $g(x) \geq 0.056$ for every $x \geq 1423.728$. We set
\begin{align*}
f(x) & = (c(x)-1)(\log^5 x - \log^4 x - \log^3x) - \log^4 x \log c(x) - (c(x)\log c(x) + 3.83c(x) - 2.65)\log^2 x \\
& \p{\q\q} - 2 \cdot 3.83c(x)\log c(x)\log x - 3.83c(x)\log^2 c(x)
\end{align*}
and substitute $c(x) = 1 + 1.1817/\log^3x$ in $f(x)$. Using the inequality $\log(1+t) \leq t$ which holds for every $t > -1$, we get $f(x) \geq
g(\log x) \geq 0.056$ for every $x \geq e^{1423.728}$. By Lemma \ref{lem401}, we obtain
\begin{displaymath}
\pi \left( x \left( 1 + \frac{1.1817}{\log^3 x} \right) \right) - \pi(x) > \frac{f(x)/\log^4(x)}{(\log (c(x)x) - 1 - \frac{1}{\log (c(x)x)} -
\frac{2.65}{\log^2 (c(x)x)})(\log x - 1 - \frac{1}{\log x} - \frac{3.83}{\log^2 x})} \geq 0
\end{displaymath}
for every $x \geq e^{1423.728}$. For every $e^{150} \leq x \leq e^{1423.728}$ the theorem follows directly from Proposition \ref{p403}. Further, we use
Proposition \ref{p402} and Proposition \ref{p404} to obtain our theorem for every $10726905041 \leq x < e^{150}$ and every $2898239 \leq x < 10726905041$,
respectively. Next, we check with a computer that
\begin{displaymath}
p_n \left( 1 + \frac{1.1817}{\log^3 p_n} \right) > p_{n+1}
\end{displaymath}
for every $\pi(58889) \leq n \leq \pi(2898239)+1$. Finally, we obtain
\begin{displaymath}
\pi \left( x + \frac{1.1817x}{\log^3 x} \right) > 5949 = \pi(x)
\end{displaymath}
for every $58837 \leq x < 58889$.
\end{proof}

\vspace{5mm}

\textsc{Mathematisches Institut, Heinrich-Heine-Universität Düsseldorf, 40225 Düsseldorf}, \textsc{Germany}

\emph{E-mail address}: \texttt{Christian.Axler@hhu.de}

\end{document}